\documentclass{birkjour}
 \newtheorem{theorem}{Theorem}[section]
 \newtheorem{corollary}[theorem]{Corollary}
 \newtheorem{lemma}[theorem]{Lemma}
 \newtheorem{proposition}[theorem]{Proposition}
 \theoremstyle{definition}
 
 \theoremstyle{remark}
 \newtheorem{remark}[theorem]{Remark}
 \newtheorem{example}{Example}
 \numberwithin{equation}{section}

\begin{document}

%
%
%
%
%
%
%
%
%

\title[New insights on slant submanifolds] {New insights on slant submanifolds}

\author[Adara M. Blaga]{Adara M. Blaga}

\address{%
Department of Mathematics\\
Faculty of Mathematics and Computer Science\\
West University of Timi\c{s}oara\\
V. P\^ arvan 4-6, Timi\c{s}oara, 300223,
Rom\^{a}nia}

\email{adarablaga@yahoo.com}



\date{8 May, 2023}

\begin{abstract}
We provide the necessary and sufficient condition for a pointwise slant submanifold with respect to two anti-commuting almost Hermitian structures to be also pointwise slant with respect to a family of almost Hermitian structures generated by them. On the other hand, we show that the property of being pointwise slant is transitive on a class of proper pointwise slant immersed submanifolds of almost Hermitian manifolds. We illustrate the results with suitable examples.
\end{abstract}

\maketitle

\section{Introduction}

The notion of \textit{slant submanifold} was initially defined in 1990 by Chen \cite{chen} as an isometric immersed submanifold of an almost Hermitian manifold with constant Wirtinger angle. In 1998, Etayo \cite{etayo} extended this notion to that of \textit{pointwise slant submanifold} (named, at that time, quasi-slant submanifold) by letting the Wirtinger angle to depend on the points of the submanifold (see also \cite{chen5}). Since then, the theory of slant and pointwise slant submanifolds has been developed in different geometries.

Apart from the ideas that have already appeared in this context, in the present paper, we put into light new aspects of pointwise slant submanifolds in the almost Hermitian setting.
On a Riemannian manifold $(\bar M,\bar g)$, we consider a family $\{\bar J_{a,b}\}_{a,b}$ of almost Hermitian structures defined by a pair of anti-commuting $\bar g$-skew-sym\-met\-ric almost complex structures $(\bar J_1,\bar J_2)$, we characterize their integrability in terms of Fr\"olicher--Nijenhuis bracket, and we provide the necessary and sufficient condition for a pointwise slant submanifold with respect to $(\bar g,\bar J_1)$ and $(\bar g,\bar J_2)$ to be also pointwise slant with respect to $(\bar g,\bar J_{a,b})$. In particular, we show that if a submanifold is anti-invariant in $(\bar M,\bar g,\bar J_1)$ and $(\bar M,\bar g,\bar J_2)$, then it is anti-invariant in $(\bar M,\bar g,\bar J_{a,b})$, too. On the other hand, we prove that a pointwise slant
immersed submanifold $M_3$ (with a slant function $\theta_3$) of a proper pointwise slant
immersed submanifold $M_2$ (with a slant function $\theta_2$) of an almost Hermitian manifold $M_1$ is also pointwise slant in $M_1$, with a slant function $\geq \max\{\theta_2,\theta_3\}$. We also generalize some of the results and construct examples.

\section{A family of almost complex structures}

Let $(\bar M,\bar g)$ be a Riemannian manifold. We shall further denote by $\bar \nabla$ the Levi-Civita connection of $\bar g$, and by $\Gamma(T\bar M)$ the set of smooth sections of $\bar M$.

We recall that an endomorphism $\bar J$ of the tangent bundle $T\bar M$ is an \textit{almost complex structure} on $\bar M$ if $\bar J^2=-I$,
where $I$ stands for the identity endomorphism. Moreover, if $\bar J$ is skew-symmetric with respect to $\bar g$, then $(\bar g,\bar J)$ is called an \textit{almost Hermitian structure} on $\bar M$. Furthermore, an almost Hermitian structure $(\bar g,\bar J)$ is said to be a \textit{K\"{a}hler structure} if $\bar J$ is parallel with respect to the Levi-Civita connection of $\bar g$, i.e., if $\bar \nabla \bar J=0$.

\bigskip

By means of two almost complex structures $\bar J_1$ and $\bar J_2$ which are anti-commuting (i.e., $\bar J_1\bar J_2=-\bar J_2\bar J_1$), we shall define a family of almost complex structures and characterize their integrability.

\begin{proposition}\label{p4}
Let $(\bar g,\bar J_i)$, $i={1,2}$, be two anti-commuting almost Her\-mi\-tian structures on $\bar M$, let $a$ and $b$ be two smooth functions on $\bar M$, and let $\bar J_{a,b}:= a\bar J_1+b\bar J_2$. Then:

(i) the endomorphism $\bar J_{a,b}$ is an almost complex structure on $\bar M$ if and only if $a^2+b^2=1$; in this case, $(\bar g,\bar J_{a,b})$ is an almost Hermitian structure on $\bar M$;

(ii) if the manifold $\bar M$ is connected, and if $\bar \nabla \bar J_1=0$ and $\bar \nabla \bar J_2=0$, then $\bar \nabla \bar J_{a,b}=0$ if and only if $a$ and $b$ are constant.
\end{proposition}
\begin{proof}
(i) $$\bar J_{a,b}^2=a^2\bar J_1^2+b^2\bar J_2^2+ab(\bar J_1\bar J_2+\bar J_2\bar J_1)=-(a^2+b^2)I,$$
and, for any $X,Y\in\Gamma(T\bar M)$, we have
\begin{align}
\bar g(\bar J_{a,b}X,Y)&=a\bar g(\bar J_1X,Y)+b\bar g(\bar J_2X,Y)\notag\\
&=-a\bar g(X,\bar J_1Y)-b\bar g(X,\bar J_2Y)=-\bar g(X,\bar J_{a,b}Y).\notag
\end{align}

(ii) For any $X,Y\in \Gamma(T\bar M)$, we have
\begin{align*}
(\bar \nabla_X \bar J_{a,b})Y:&=\bar \nabla_X\bar J_{a,b}Y-\bar J_{a,b}(\bar \nabla_XY)\notag\\
&=a(\bar \nabla_X \bar J_1)Y+b(\bar \nabla_X \bar J_2)Y+X(a)\bar J_1Y+X(b)\bar J_2Y\notag\\
&=X(a)\bar J_1Y+X(b)\bar J_2Y;
\end{align*}
therefore, $\bar \nabla\bar J_{a,b}=0$ if and only if
$$X(a)\bar J_1Y+X(b)\bar J_2Y=0$$
for any $X,Y\in \Gamma(T\bar M)$.
Applying $\bar J_1$ and $\bar J_2$, respectively, and taking into account that $\bar J_1\bar J_2=-\bar J_2\bar J_1$, we get
$$\Big(X(a)\Big)^2+\Big(X(b)\Big)^2=0$$
for any $X\in \Gamma(T\bar M)$, hence the conclusion.
The converse implication is trivial.
\end{proof}

More general, we have

\begin{proposition}
Let $(\bar g,\bar J_i)$, $i=\overline{1,k}$, be $k$ pairwise anti-commuting almost Hermitian structures on $\bar M$, let $a_i$, $i=\overline{1,k}$, be $k$ smooth functions on $\bar M$, and let $\bar J_{a_1,\dots,a_k}:=\nolinebreak \sum_{i=1}^ka_i\bar J_i$.
Then, the endomorphism $\bar J_{a_1,\dots,a_k}$ is an almost complex structure on $\bar M$ if and only if $\sum_{i=1}^ka_i^2=1$. In this case, $(\bar g,\bar J_{a_1,\dots,a_k})$ is an almost Hermitian structure on $\bar M$.
\end{proposition}
\begin{proof}
$$\bar J_{a_1,\dots,a_k}^2=\sum_{i=1}^ka_i^2\bar J_i^2+\sum_{1\leq i<j\leq k}a_ia_j(\bar J_i\bar J_j+\bar J_j\bar J_i)=-\sum_{i=1}^ka_i^2I,$$
and, for any $X,Y\in \Gamma(T\bar M)$, we have
\begin{align*}
\bar g(\bar J_{a_1,\dots,a_k}X,Y)&=\sum_{i=1}^ka_i\bar g(\bar J_iX,Y)\\
&=-\sum_{i=1}^ka_i\bar g(X,\bar J_iY)=-\bar g(X,\bar J_{a_1,\dots,a_k}Y).
\end{align*}
\end{proof}

Let us denote by $N_{\bar J_i}$ the \textit{Nijenhuis tensor field} of $\bar J_i$, $i={1,2}$:
$$N_{\bar J_i}(X,Y):=[\bar J_iX,\bar J_iY]-\bar J_i[\bar J_iX, Y]-\bar J_i[X, \bar J_iY]+\bar J_i^{2}[X,Y],$$
and by $[\bar J_1,\bar J_2]$ the \textit{Fr\"olicher--Nijenhuis bracket} of $\bar J_1$ and $\bar J_2$:
\begin{align}
[\bar J_1,\bar J_2](X,Y):&=[\bar J_1X,\bar J_2Y]+[\bar J_2X,\bar J_1Y]+\bar J_1\bar J_2[X,Y]+\bar J_2\bar J_1[X,Y]\notag\\
&\hspace{8pt}-\bar J_1[\bar J_2X,Y]-\bar J_1[X,\bar J_2Y]-\bar J_2[\bar J_1X,Y]-\bar J_2[X,\bar J_1Y]\notag
\end{align}
for $X,Y\in \Gamma(T\bar M)$.
Since $\bar J_1$ and $\bar J_2$ anti-commute, we have
\begin{align}
[\bar J_1,\bar J_2](X,Y)&=[\bar J_1X,\bar J_2Y]+[\bar J_2X,\bar J_1Y]\notag\\
&\hspace{10pt}-\bar J_1\Big([\bar J_2X,Y]+[X,\bar J_2Y]\Big)-\bar J_2\Big([\bar J_1X,Y]+[X,\bar J_1Y]\Big).\notag
\end{align}

We will further characterize the integrability of $\bar J_{a,b}$ in terms of Fr\"olicher--Nijenhuis bracket of $\bar J_1$ and $\bar J_2$.

We recall that a $(1,1)$-tensor field $\bar J$ on $\bar M$ is said to be \textit{integrable} if its Nijenhuis tensor field, $N_{\bar J}$, vanishes.
In particular, an integrable almost complex structure is called a \textit{complex structure}.

\begin{proposition}
Let $(\bar J_1,\bar J_2)$ be a pair of anti-commuting almost complex structures on $\bar M$, and let $\bar J_{a,b}:=a\bar J_1+b\bar J_2$ with $a^2+b^2=1$, $a,b\in \mathbb R$. Then, for any $X,Y\in \nolinebreak \Gamma(T\bar M)$, we have
$$N_{\bar J_{a,b}}(X,Y)=a^2N_{\bar J_1}(X,Y)+b^2N_{\bar J_2}(X,Y)+ab[\bar J_1,\bar J_2](X,Y).$$
\end{proposition}
\begin{proof}
The relation follows by a straightforward computation, considering the anti-symmetry and the $\mathbb R$-bilinearity of the Lie bracket.
\end{proof}

Now we can state
\begin{proposition}
If $(\bar J_1,\bar J_2)$ is a pair of anti-commuting complex structures on $\bar M$, and $\bar J_{a,b}:=a\bar J_1+b\bar J_2$ with $a^2+b^2=1$, $a,b\in \mathbb R^*$, then, $\bar J_{a,b}$ is integrable if and only if the Fr\"olicher--Nijenhuis bracket of $\bar J_1$ and $\bar J_2$ vanishes.
\end{proposition}

\section{Pointwise slant submanifolds with respect to a family of almost Hermitian manifolds}

Let $M$ be an immersed submanifold of an almost Hermitian manifold $(\bar M,\bar g,\bar J_i)$, $i=\nolinebreak {1,2}$. Throughout the paper, we shall assume all the immersions to be injective. Also, we shall use the same notation for a metric on a manifold as for the induced metric on a submanifold.

We have the orthogonal decomposition
$$T\bar M=TM\oplus T^{\bot} M$$
and, for any $X\in \Gamma(TM)$, we denote
$$\bar J_iX=T_iX+N_iX, \ \ i={1,2},$$
where $T_iX$ and $N_iX$ represent the tangential and the normal component of $\bar J_iX$, respectively.

We recall \cite{chen5,etayo} that $M$ is called a \textit{pointwise slant submanifold} of $(\bar M,\bar g,\bar J_i)$ with the slant function $\theta_i$, $\theta_i(x)\in (0,\frac{\pi}{2}]$ for any $x\in M$, if,
for any $X\in\nolinebreak \Gamma(TM)$ with $X_x\neq 0$, the angle between $\bar J_iX_x$ and $T_xM$ does not depend on the nonzero tangent vector $X_x$, but only on the point $x$ of $M$.
In this case, for any $X\in \Gamma(TM)$, we have
$$\Vert T_iX\Vert^2=\cos^2\theta_i\Vert X\Vert^2.$$
Moreover, if $\theta_i(x)\neq \frac{\pi}{2}$ for any $x\in M$, then $M$ is called a \textit{proper pointwise slant submanifold}.
We remark \cite{latcu} that the slant function of a proper pointwise slant submanifold is smooth.

Furthermore, if the angle between $\bar J_iX_x$ and $T_x M$ does not depend on the nonzero tangent vector $X_x$, neither on the point $x$ of $M$, then $M$ is called a \textit{slant submanifold} of $(\bar M,\bar g,\bar J_i)$ with the slant angle $\theta_i$ (respectively, a \textit{proper slant submanifold} if $\theta_i\neq \frac{\pi}{2}$, and an \textit{anti-invariant submanifold} if $\theta_i=\frac{\pi}{2}$).

\bigskip

The necessary and sufficient condition for a pointwise slant submanifold with respect to $(\bar g,\bar J_1)$ and to $(\bar g,\bar J_2)$, to be also pointwise slant with respect to $(\bar g,\bar J_{a,b})$, will be provided by the next theorem.
\begin{theorem}\label{p5}
Let $(\bar g,\bar J_i)$, $i={1,2}$, be two anti-commuting almost Her\-mi\-tian structures on $\bar M$, let $M$ be a pointwise slant submanifold of $(\bar M,\bar g,\bar J_i)$, with the slant function $\theta_i$, $i={1,2}$, and let
$\bar J_{a,b}:=a\bar J_1+b\bar J_2$ for $a$ and $b$ smooth functions on $\bar M$ with $a^2+b^2=1$.
Then, $M$ is a pointwise slant submanifold of the almost Hermitian manifold $(\bar M,\bar g,\bar J_{a,b})$ if and only if,
for any $x\in M$ with $a(x)b(x)\neq 0$, $\bar g(T_1u,T_2u)$ does not depend on the unitary tangent vector $u\in T_xM$.

In this case, the slant function $\theta$ is given by
$$\theta(x)=\arccos\Big(\sqrt{a^2(x)\cos^2\theta_1(x)+b^2(x)\cos^2\theta_2(x)+2a(x)b(x)\bar g(T_1u,T_2u)}\Big)$$
for any $x\in M$ and any $u\in T_xM$ with $\Vert u\Vert=1$.
\end{theorem}
\begin{proof}
For any $X\in \Gamma(TM)$, we have
$$\Vert T_iX\Vert^2=\cos^2\theta_i\Vert X\Vert^2, \ \ i={1,2}.$$

Let $\bar J_{a,b}X=T_{a,b}X+N_{a,b}X$ with $T_{a,b}X\in \Gamma(TM)$ and $N_{a,b}X\in \Gamma(T^{\bot}M)$. Then,
$$\left\{
    \begin{array}{ll}
      T_{a,b}X=aT_1X+bT_2X \\
      N_{a,b}X=aN_1X+bN_2X
    \end{array},
  \right.$$
which imply
\begin{align}
\Vert T_{a,b}X\Vert^2&=a^2\Vert T_1X\Vert^2+b^2\Vert T_2X\Vert^2+2ab \cdot \bar g(T_1X,T_2X)\notag\\
&=(a^2\cos^2\theta_1+b^2\cos^2\theta_2)\Vert X\Vert^2+2ab \cdot \bar g(T_1X,T_2X).\notag
\end{align}

If $M$ is a pointwise slant submanifold of $(\bar M,\bar g,\bar J_{a,b})$, with the slant function $\theta$, then, for any
$x\in M$ with $a(x)b(x)\neq 0$, and any $v\in T_xM\setminus\{0\}$, we have
\begin{align}
\cos^2\theta(x) \Vert v\Vert^2&=\Vert T_{a,b}v\Vert^2&\notag\\
&\hspace{-20pt}=\left(a^2(x)\cos^2\theta_1(x)+b^2(x)\cos^2\theta_2(x)+2a(x)b(x)\frac{\bar g(T_1v,T_2v)}{\Vert v\Vert^2}\right)\Vert v\Vert^2,\notag
\end{align}
from where, for any unitary tangent vector $u\in T_xM$,
\begin{align}
\cos^2\theta(x)&=a^2(x)\cos^2\theta_1(x)+b^2(x)\cos^2\theta_2(x)+2a(x)b(x)\bar g(T_1u,T_2u),\notag
\end{align}
and so $\bar g(T_1u,T_2u)$ does not depend on $u$.
Conversely, if, for any $x\in M$ with $a(x)b(x)\neq \nolinebreak 0$, $\bar g(T_1u,T_2u)$ does not depend on the unitary tangent vector $u\in T_xM$, then, for any $x\in M$ and any $v\in T_xM\setminus\{0\}$, $a(x)b(x)\frac{\displaystyle \bar g(T_1v,T_2v)}{\displaystyle \Vert v\Vert^2}$ does not depend on \nolinebreak $v$, so
\begin{align*}
\frac{\Vert T_{a,b}v\Vert^2}{\Vert J_{a,b}v\Vert^2}&=\frac{\Vert T_{a,b}v\Vert^2}{\Vert v\Vert^2}&\notag\\
&\hspace{-20pt}=a^2(x)\cos^2\theta_1(x)+b^2(x)\cos^2\theta_2(x)+2a(x)b(x)\frac{\bar g(T_1v,T_2v)}{\Vert v\Vert^2},\notag
\end{align*}
does not depend on $v\in T_xM\setminus\{0\}$; hence, $M$ is a pointwise slant submanifold of $(\bar M,\bar g,\bar J_{a,b})$.
\end{proof}

\begin{example} \label{e1}
Let $\bar J_1$ and $\bar J_2$ be two anti-commuting and $\bar g_0$-skew symmetric almost complex structures on the $8$-dimensional Euclidean space $\mathbb R^{8}$ (with the Euclidean metric $\bar g_0$), defined by
$$\bar J_1\Big(\frac{\partial}{\partial u_i}\Big)=-\frac{\partial}{\partial u_{i+2}}, \ \
\bar J_1\Big(\frac{\partial}{\partial u_{i+2}}\Big)=\frac{\partial}{\partial u_i},$$
$$\bar J_1\Big(\frac{\partial}{\partial v_i}\Big)=-\frac{\partial}{\partial v_{i+2}}, \ \
\bar J_1\Big(\frac{\partial}{\partial v_{i+2}}\Big)=\frac{\partial}{\partial v_i}$$
for $i={1,2}$, and by
$$\bar J_2\Big(\frac{\partial}{\partial u_i}\Big)=-\frac{\partial}{\partial v_i}, \ \
\bar J_2\Big(\frac{\partial}{\partial v_i}\Big)=\frac{\partial}{\partial u_i}$$
for $i={1,2}$, and
$$\bar J_2\Big(\frac{\partial}{\partial u_i}\Big)=\frac{\partial}{\partial v_i}, \ \
\bar J_2\Big(\frac{\partial}{\partial v_i}\Big)=-\frac{\partial}{\partial u_i}$$
for $i={3,4}$,
where we denoted by $(u_1,v_1,\dots,u_4,v_4)$ the canonical coordinates in $\mathbb R^8$. Let $M$ be the immersed submanifold of $\mathbb R^{8}$ defined by
$$F:\mathbb R^2\rightarrow \mathbb R^8, \ \ F(x_1,x_2):= \Big(2x_{1}, x_1,x_1^2, x_1+x_2, x_1-x_2,2x_2, x_2,x_2^2\Big).$$

Then, $TM$ is spanned by the orthogonal vector fields:
$$X_{1} =2\frac{\partial }{\partial u_{1}}+\frac{\partial
}{\partial v_{1}}+2x_1\frac{\partial }{\partial u_{2}}+\frac{\partial }{\partial v_{2}}+\frac{\partial }{\partial u_{3}},$$
$$X_{2} =\frac{\partial }{\partial v_{2}}-\frac{\partial }{\partial u_{3}}+2\frac{\partial
}{\partial v_{3}}+\frac{\partial }{\partial u_{4}}+2x_2\frac{\partial }{\partial v_{4}}.$$

Applying $\bar J_1$ and $\bar J_2$ to the base vector fields of $TM$, we get:
\begin{align}
\bar J_1 X_{1} &=\frac{\partial }{\partial u_{1}}-2\frac{\partial }{\partial u_{3}}-\frac{\partial
}{\partial v_{3}}-2x_1\frac{\partial }{\partial u_{4}}-\frac{\partial }{\partial v_{4}},\notag \\
\bar J_1 X_{2} &=-\frac{\partial }{\partial u_{1}}+2\frac{\partial
}{\partial v_{1}}+\frac{\partial }{\partial u_{2}}+2x_2\frac{\partial }{\partial v_{2}}-\frac{\partial }{\partial v_{4}},\notag \\
\bar J_2X_{1} &=\frac{\partial }{\partial u_{1}}-2\frac{\partial }{\partial v_{1}}+\frac{\partial }{\partial u_{2}}-2x_{1}\frac{\partial }{\partial v_{2}}+\frac{\partial }{\partial v_{3}},\notag \\
\bar J_2X_{2} &=\frac{\partial }{\partial u_{2}}-2\frac{\partial }{\partial u_{3}}-\frac{\partial }{\partial v_{3}}-2x_2\frac{\partial }{\partial
u_4}+\frac{\partial }{\partial v_{4}}.\notag
\end{align}

Then,
$$\Vert X_1\Vert=\Vert \bar J_1X_1\Vert=\Vert \bar J_2X_1\Vert=\sqrt{4x_1^2+7},$$
$$\Vert X_2\Vert=\Vert \bar J_1X_2\Vert=\Vert \bar J_2X_2\Vert=\sqrt{4x_2^2+7},$$
$$\frac{\left\vert \bar g_0(\bar J_1 X_1,X_2) \right\vert }{%
\left\Vert \bar J_1 X_1\right\Vert \cdot\left\Vert X_{2}\right\Vert }=\frac{2\vert x_1+x_2\vert }{\sqrt{\left(4x_1^2+7\right)\left(4x_2^2+7\right)}}=\frac{\left\vert \bar g_0(\bar J_1 X_2,X_1) \right\vert }{\left\Vert \bar J_1 X_2\right\Vert \cdot\left\Vert X_1\right\Vert };$$
hence, $M$ is a pointwise slant submanifold of $(\mathbb R^8, \bar g_0,\bar J_1)$, with the slant function $$\theta_{1}(x_{1},x_{2}) =\arccos \left(\frac{2\vert x_1+x_2\vert}{\sqrt{\left(4x_1^2+7\right)\left(4x_2^2+7\right)}}\right).$$

On the other hand,
$$\frac{\left\vert \bar g_0(\bar J_2 X_1,X_2) \right\vert }{%
\left\Vert \bar J_2 X_1\right\Vert \cdot\left\Vert X_{2}\right\Vert }=\frac{2\vert x_1-1\vert }{\sqrt{\left(4x_1^2+7\right)\left(4x_2^2+7\right)}}=\frac{\left\vert \bar g_0(\bar J_2 X_2,X_1) \right\vert }{\left\Vert \bar J_2 X_2\right\Vert \cdot\left\Vert X_1\right\Vert };$$
hence, $M$ is a pointwise slant submanifold of $(\mathbb R^8, \bar g_0,\bar J_2)$, with the slant function $$\theta_{2}(x_{1},x_{2}) =\arccos \left(\frac{2\vert x_1-1\vert}{\sqrt{\left(4x_1^2+7\right)\left(4x_2^2+7\right)}}\right).$$

We define $\bar J_{a,b}:=a\bar J_1+b\bar J_2$ for $a^2+b^2=1$. Then,
\begin{align}
\bar J_{a,b} X_{1} &=(a+b)\frac{\partial }{\partial u_{1}}-2b\frac{\partial
}{\partial v_{1}}+b\frac{\partial }{\partial u_{2}}-2bx_1\frac{\partial }{\partial v_{2}}\notag \\
&\hspace{10pt}-2a\frac{\partial
}{\partial u_{3}}+(b-a)\frac{\partial }{\partial v_{3}}
-2ax_1\frac{\partial }{\partial u_{4}}-a\frac{\partial }{\partial v_{4}},\notag
\end{align}
\begin{align}
\bar J_{a,b} X_{2} &=-a\frac{\partial }{\partial u_{1}}+2a\frac{\partial
}{\partial v_{1}}+(a+b)\frac{\partial }{\partial u_{2}}+2ax_2\frac{\partial }{\partial v_{2}}\notag \\
&\hspace{10pt}-2b\frac{\partial
}{\partial u_{3}}-b\frac{\partial }{\partial v_{3}}-2bx_2\frac{\partial }{\partial u_{4}}+(b-a)\frac{\partial }{\partial v_{4}},\notag
\end{align}
$$\Vert \bar J_{a,b}X_1\Vert=\sqrt{4x_1^2+7}, \ \ \Vert \bar J_{a,b}X_2\Vert=\sqrt{4x_2^2+7},$$
$$\frac{\left\vert \bar g_0(\bar J_{a,b} X_1,X_2) \right\vert }{%
\left\Vert \bar J_{a,b} X_1\right\Vert \cdot\left\Vert X_{2}\right\Vert }=\frac{2\vert a(x_1+x_2)+b(x_1-1)\vert }{\sqrt{\left(4x_1^2+7\right)\left(4x_2^2+7\right)}}=\frac{\left\vert \bar g_0(\bar J_{a,b} X_2,X_1) \right\vert }{\left\Vert \bar J_{a,b} X_2\right\Vert \cdot\left\Vert X_1\right\Vert };$$
hence, $M$ is a pointwise slant submanifold of $(\mathbb R^8, \bar g_0,\bar J_{a,b})$, with the slant function $$\theta(x_{1},x_{2}) =\arccos \left(\frac{2\vert a(x_1+x_2)+b(x_1-1)\vert }{\sqrt{\left(4x_1^2+7\right)\left(4x_2^2+7\right)}}\right),$$
and we notice that
$$\cos\theta=\sqrt{a^2\cos^2\theta_1+b^2\cos^2\theta_2+2ab\cdot \frac{4(x_1+x_2)(x_1-1)}{\left(4x_1^2+7\right)\left(4x_2^2+7\right)}}.$$
\end{example}

\begin{corollary}\label{c9}
Under the hypothesis of Theorem \ref{p5}, if $\bar g(T_1X,T_2X)=0$ for any $X\in  \Gamma(TM)$,
then $M$ is a pointwise slant submanifold of
$(\bar M,\bar g,\bar J_{a,b})$, with the slant function
$$\theta=\arccos\Big(\sqrt{a^2\cos^2\theta_1+b^2\cos^2\theta_2}\Big).$$
\end{corollary}

\begin{remark} Under the hypothesis of Theorem \ref{p5}:

(i) if $\theta_1=\theta_2=\frac{\pi}{2}$, then $\theta=\frac{\pi}{2}$;

(ii) if $\bar g(T_1X,T_2X)=0$ for any $X\in \Gamma(TM)$,
since $$\min \{\cos^2\theta_1,\cos^2\theta_2\}\leq a^2\cos^2\theta_1+(1-a^2)\cos^2\theta_2\leq \max \{\cos^2\theta_1,\cos^2\theta_2\},$$ we notice that
$$\min \{\theta_1,\theta_2\} \leq \theta \leq \max \{\theta_1,\theta_2\}.$$
\end{remark}

For $\theta_2=\frac{\pi}{2}$, we get
\begin{corollary}\label{c1}
Let $(\bar g,\bar J_i)$, $i={1,2}$, be two anti-commuting almost Her\-mi\-tian structures on $\bar M$, let $M$ be a pointwise slant submanifold of $(\bar M,\bar g,\bar J_1)$, with the slant function $\theta_1$, and let $M$ be an anti-invariant submanifold of $(\bar M,\bar g,\bar J_2)$. Then, $M$ is a pointwise slant submanifold of the almost Hermitian manifold $(\bar M,\bar g,\bar J_{a,b})$, with the slant function $$\theta=\arccos\Big(|a|\cos\theta_1\Big),$$ where $\bar J_{a,b}:=a\bar J_1+b\bar J_2$ for $a$ and $b$ smooth functions on $\bar M$ with $a^2+b^2=1$.
\end{corollary}

\begin{remark}
Under the hypothesis of Corollary \ref{c1},
since $|a|\cos\theta_1\leq \cos\theta_1$, we notice that $\theta\geq  \theta_1$.
\end{remark}

\begin{example}
Let $(\bar g_0,\bar J_1)$ and $(\bar g_0,\bar J_2)$ be the almost Hermitian structures on the $8$-dimensional Euclidean space $\mathbb R^{8}$ defined in Example \ref{e1}. We will denote by $(u_1,v_1,\dots,u_4,v_4)$ the canonical coordinates in $\mathbb R^{8}$. Let $M$ be the immersed submanifold of $\mathbb R^{8}$ defined by
$$F:\mathbb R^2\rightarrow \mathbb R^8, \ \ F(x_1,x_2):= \Big(2x_{1}, x_1,x_1^2, 1, 2x_2, x_2,x_2^2,1\Big).$$

Then, $TM$ is spanned by the orthogonal vector fields:
$$X_{1} =2\frac{\partial }{\partial u_{1}}+\frac{\partial
}{\partial v_{1}}+2x_1\frac{\partial }{\partial u_{2}}, \ \
X_{2} =2\frac{\partial }{\partial u_{3}}+\frac{\partial
}{\partial v_{3}}+2x_2\frac{\partial }{\partial u_{4}}.$$

Applying $\bar J_1$ and $\bar J_2$, respectively, to the base vector fields of $TM$, we get:
$$\bar J_1 X_{1} =-2\frac{\partial }{\partial u_{3}}-\frac{\partial
}{\partial v_{3}}-2x_1\frac{\partial }{\partial u_{4}}, \ \
\bar J_1 X_{2} =2\frac{\partial }{\partial u_{1}}+\frac{\partial
}{\partial v_{1}}+2x_2\frac{\partial }{\partial u_{2}},$$
$$\bar J_2X_{1} =\frac{\partial }{\partial u_{1}}-2\frac{\partial }{\partial v_{1}}-2x_{1}\frac{\partial }{\partial v_{2}}, \ \
\bar J_2X_{2} =-\frac{\partial }{\partial u_{3}}+2\frac{\partial }{\partial v_{3}}+2x_2\frac{\partial }{\partial
v_4}.$$

Then,
$$\Vert X_1\Vert=\Vert \bar J_1X_1\Vert=\Vert \bar J_2X_1\Vert=\sqrt{4x_1^2+5},$$
$$\Vert X_2\Vert=\Vert \bar J_1X_2\Vert=\Vert \bar J_2X_2\Vert=\sqrt{4x_2^2+5},$$
$$\frac{\left\vert \bar g_0(\bar J_1 X_1,X_2) \right\vert }{%
\left\Vert \bar J_1 X_1\right\Vert \cdot\left\Vert X_{2}\right\Vert }=\frac{\vert 4x_1x_2+5\vert }{\sqrt{\left(4x_1^2+5\right)\left(4x_2^2+5\right)}}=\frac{\left\vert \bar g_0(\bar J_1 X_2,X_1) \right\vert }{\left\Vert \bar J_1 X_2\right\Vert \cdot\left\Vert X_1\right\Vert };$$
hence, $M$ is a pointwise slant submanifold of $(\mathbb R^8, \bar g_0,\bar J_1)$, with the slant function $$\theta_{1}(x_{1},x_{2}) =\arccos \left(\frac{\vert 4x_1x_2+5\vert}{\sqrt{\left(4x_1^2+5\right)\left(4x_2^2+5\right)}}\right).$$

On the other hand,
$$\bar g_0(\bar J_2 X_1,X_2) =0=\bar g_0(\bar J_2 X_2,X_1);$$
hence, $M$ is an anti-invariant submanifold of $(\mathbb R^8, \bar g_0,\bar J_2)$.

We define $\bar J_{a,b}:=a\bar J_1+b\bar J_2$ for $a^2+b^2=1$. Then,
$$\bar J_{a,b} X_{1} =b\frac{\partial }{\partial u_{1}}-2b\frac{\partial
}{\partial v_{1}}-2bx_1\frac{\partial }{\partial v_{2}}-2a\frac{\partial
}{\partial u_{3}}-a\frac{\partial
}{\partial v_{3}}-2ax_1\frac{\partial
}{\partial u_{4}},$$
$$\bar J_{a,b} X_{2} =2a\frac{\partial }{\partial u_{1}}+a\frac{\partial
}{\partial v_{1}}+2ax_2\frac{\partial }{\partial u_{2}}-b\frac{\partial
}{\partial u_{3}}+2b\frac{\partial
}{\partial v_{3}}+2bx_2\frac{\partial
}{\partial v_{4}},$$
$$\Vert \bar J_{a,b}X_1\Vert=\sqrt{4x_1^2+5}, \ \ \Vert \bar J_{a,b}X_2\Vert=\sqrt{4x_2^2+5},$$
$$\frac{\left\vert \bar g_0(\bar J_{a,b} X_1,X_2) \right\vert }{%
\left\Vert \bar J_{a,b} X_1\right\Vert \cdot\left\Vert X_{2}\right\Vert }=\frac{\vert a\vert\cdot\vert 4x_1x_2+5\vert }{\sqrt{\left(4x_1^2+5\right)\left(4x_2^2+5\right)}}=\frac{\left\vert \bar g_0(\bar J_{a,b} X_2,X_1) \right\vert }{\left\Vert \bar J_{a,b} X_2\right\Vert \cdot\left\Vert X_1\right\Vert };$$
hence, $M$ is a pointwise slant submanifold of $(\mathbb R^8, \bar g_0,\bar J_{a,b})$, with the slant function $$\theta(x_{1},x_{2}) =\arccos \left(\vert a\vert\cdot\frac{\vert 4x_1x_2+5\vert}{\sqrt{\left(4x_1^2+5\right)\left(4x_2^2+5\right)}}\right).$$
\end{example}

More general, we have
\begin{theorem}\label{p6}
Let $(\bar g,\bar J_i)$, $i=\overline{1,k}$, be $k$ pairwise anti-commuting almost Hermitian structures on $\bar M$, let
$M$ be a pointwise slant submanifold of $(\bar M,\bar g,\bar J_i)$, with the slant function $\theta_i$, $i=\overline{1,k}$, and let
$\bar J_{a_1,\dots,a_k}:=\sum_{i=1}^ka_i\bar J_i$ for $a_i$, $i=\overline{1,k}$, smooth functions on $\bar M$ with $\sum_{i=1}^ka_i^2=1$.
Then, $M$ is a pointwise slant submanifold of the almost Hermitian manifold
$(\bar M,\bar g,\bar J_{a_1,\dots,a_k})$ if and only if,
for any $x\in M$,
$\sum_{1\leq i<j\leq k}a_i(x)a_j(x)\bar g(T_iu,T_ju)$ does not depend on the unitary tangent vector $u\in \nolinebreak T_xM$.

In this case, the slant function $\theta$ is given by
$$\theta(x)=\arccos\left(\sqrt{\sum_{i=1}^ka_i^2(x)\cos^2\theta_i(x)+2\sum_{1\leq i<j\leq k}a_i(x)a_j(x)\bar g(T_iu,T_ju)}\right)$$
for any $x\in M$ and any $u\in T_xM$ with $\Vert u\Vert=1$.
\end{theorem}

\begin{corollary}\label{c10}
Under the hypothesis of Theorem \ref{p6}, if $\bar g(T_iX,T_jX)=0$ for any $X\in \nolinebreak \Gamma(TM)$ and any $i\neq j$, then $M$ is a pointwise slant submanifold of
$(\bar M,\bar g,\bar J_{a_1,\dots,a_k})$, with the slant function
$$\theta=\arccos\left(\sqrt{\sum_{i=1}^ka_i^2\cos^2\theta_i}\right).$$
\end{corollary}

\begin{remark}
Under the hypothesis of Theorem \ref{p6}:

(i) if $\theta_i=\frac{\pi}{2}$ for any $i\in\{1,\dots,k\}$, then $\theta=\frac{\pi}{2}$;

(ii) if $\bar g(T_iX,T_jX)=0$ for any $X\in \Gamma(TM)$ and any $i\neq j$,
since $$\min\{\cos^2\theta_i\ : \ i=\overline{1,k}\}\leq\sum_{i=1}^ka_i^2\cos^2\theta_i\leq\max\{\cos^2\theta_i\ : \ i=\overline{1,k}\},$$ we notice that $$\min\{\theta_1,\dots,\theta_k\}\leq\theta\leq \max\{\theta_1,\dots,\theta_k\}.$$
\end{remark}

Let $k\in \mathbb N$, $k\geq 2$, and let $(\bar M,\bar g, \bar J_i)$ and $(\bar M_i,\bar g_i, \bar J_i)$, $i=\overline{1,k}$, be almost Hermitian manifolds. We consider the direct products
$$\Big(\tilde {\bar M}:=\bar M\times \stackrel{k}{\dots}\times \bar M, \ \tilde {\bar g}\left((X_1,\dots,X_k), (Y_1,\dots,Y_k)\right):=\sum_{i=1}^k\bar g(X_i,Y_i)\Big),$$
$$\Big(\hat {\bar M}:=\bar M_1\times {\cdots}\times \bar M_k,\ \hat {\bar g}\left((X_1,\dots,X_k), (Y_1,\dots,Y_k)\right):=\sum_{i=1}^k\bar g_i(X_i,Y_i)\Big),$$
and the endomorphism
$\tilde {\bar J}(X_1,\dots,X_k):=(\bar J_1X_1,\dots,\bar J_kX_k)$.
Then, $(\tilde {\bar M}, \tilde {\bar g}, \tilde {\bar J})$ and $(\hat {\bar M}, \hat{\bar g}, \tilde {\bar J})$ are almost Hermitian manifolds, and we can state

\begin{proposition}
(i) If $M$ is a pointwise slant submanifold of the almost Hermitian manifolds $(\bar M,\bar g,\bar J_i)$, $i=\overline{1,k}$, with distinct slant functions $\theta_i$, $i=\overline{1,k}$, then $\tilde M:=\nolinebreak M\times \stackrel{k}{\dots}\times M$ is a $k$-pointwise slant submanifold of $(\tilde {\bar M},\tilde {\bar g},\tilde {\bar J})$, with the slant functions $\theta_i$, $i=\overline{1,k}$.

(ii) If $M_i$ is a pointwise slant submanifold of the almost Hermitian manifold $(\bar M_i,\bar g_i,\bar J_i)$, with the slant function $\theta_i$, $i=\overline{1,k}$, then
$\hat M:=M_1\times {\cdots}\times M_k$ is a pointwise slant submanifold of $(\hat {\bar M},\hat {\bar g},\tilde {\bar J})$, with a slant function $\theta$,
if and only if, $\theta$ and $\theta_i$, for all $i=\overline{1,k}$, are constant and equal to the same value.
\end{proposition}
\begin{proof}
(i) follows immediately from the definition of a \textit{$k$-pointwise slant submanifold of an almost Hermitian manifold} (see \cite{latcu2}).

(ii) Let $X_i\in \Gamma(TM_i)$, $i=\overline{1,k}$. We denote by $T_iX_i$ the tangential component of $\bar J_iX_i$ ($T_iX_i\in \Gamma(TM_i)$).
For any $(x_1,\dots,x_k)\in M_1\times\cdots \times M_k$, we identify $T_{(x_1,\dots,x_k)}(M_1\times \cdots \times M_k)$ with $T_{x_1}M_1\oplus \cdots \oplus T_{x_k}M_k$ and further,
$$T(M_1\times \cdots \times M_k)\cong \pi^*_1(TM_1)\oplus \cdots \oplus\pi^*_k(TM_k),$$
where $\pi_i:M_1\times \cdots \times M_k\rightarrow M_i$, $i=\overline{1,k}$, are the canonical projections.

Let $\tilde X\in \Gamma(T\hat M)$. Then, $\tilde X=(X_1, \dots,X_k)$, $X_i\in \Gamma(TM_i)$, $i=\overline{1,k}$.
We denote by $\tilde T\tilde X$ the tangential component of $\tilde{\bar{J}}\tilde X$ ($\tilde T\tilde X\in \Gamma(T\hat M)$), and we have
$\tilde T\tilde X=(T_1X_1,\dots,T_kX_k)$. Then, we get
$$\Vert \tilde T\tilde X\Vert^2_{\hat {\bar g}}=\sum_{i=1}^k\Vert T_iX_i\Vert^2_{\bar g_i}=\sum_{i=1}^k\cos^2\theta_i\cdot \Vert X_i\Vert^2_{\bar g_i}, \ \
\Vert \tilde X\Vert^2_{\hat {\bar g}}=\sum_{i=1}^k\Vert X_i\Vert^2_{\bar g_i}\,;$$
hence, $\hat M$ is a pointwise slant submanifold of $(\hat {\bar M},\hat {\bar g},\tilde {\bar J})$, with a slant function $\theta$, if and only if
$$\cos^2\theta \sum_{i=1}^k\Vert X_i\Vert^2_{\bar g_i}=\sum_{i=1}^k\cos^2\theta_i\cdot \Vert X_i\Vert^2_{\bar g_i}$$
for any $X_i\in \Gamma(TM_i)$, $i=\overline{1,k}$, equivalent to
$$\sum_{i=1}^k(\cos^2\theta-\cos^2\theta_i)\Vert X_i\Vert^2_{\bar g_i}=0$$
for any $X_i\in \Gamma(TM_i)$, $i=\overline{1,k}$, and we get the conclusion.
\end{proof}

\section{Slant submanifolds of slant submanifolds}

In this section, we put into light the transitivity property of "being pointwise slant submanifold" on a class of proper pointwise slant
immersed submanifolds in the almost Hermitian setting.

\bigskip

Let $M_2$ be an
immersed submanifold of an almost Hermitian manifold $(M_1,g,J_1)$, and let $\nabla^1$ and $\nabla^2$ be the Levi-Civita connections on $M_1$ and $M_2$, respectively.
For any $x\in M_2$, we have the orthogonal decomposition
$$T_xM_1=T_xM_2\oplus T^{\bot}_x M_2,$$
and for any $v\in T_xM_2$, we denote
$$J_1v=T_1v+N_1v,$$
where $T_1v$ and $N_1v$ represent the tangential and the normal component of $J_1v$, respectively.

\begin{lemma}
Let $M_2$ be a proper pointwise slant
immersed submanifold of an almost Hermitian manifold $(M_1,g,J_1)$, with the slant function $\theta_1$, and let
\linebreak
$J_2:=\sec\theta_1\cdot T_1$. Then, $(g, J_2)$ is an almost Hermitian structure on $M_2$. Moreover, it is a K\"{a}hler structure if and only if $$(\nabla^2_XT_1)Y=-\tan \theta_1\cdot X(\theta_1)\cdot T_1Y$$
for any $X,Y\in \Gamma(TM_2)$.
\end{lemma}
\begin{proof}
For any $X,Y\in \Gamma(TM_2)$, we have
\begin{align*}
J_2^2X&=\sec^2\theta_1(-\cos^2\theta_1)X=-X,\\
g(J_2X,Y)&=\sec\theta_1\cdot g(T_1X,Y)=-\sec\theta_1\cdot g(X,T_1Y)=-g(X,J_2Y),\\
(\nabla^2_XJ_2)Y&=\nabla^2_XJ_2Y-J_2(\nabla^2_XY)=X(\sec\theta_1)T_1Y+\sec\theta_1\cdot (\nabla^2_XT_1)Y,
\end{align*}
hence the conclusion.
\end{proof}

We shall further call $J_2$ the \textit{naturally induced almost complex structure} on $M_2$.

\bigskip

\begin{theorem}\label{p2}
Let $M_2$ be a proper pointwise slant
immersed submanifold of an almost Hermitian manifold $(M_1,g,J_1)$, with the slant function $\theta_1$, let
$J_2:=\sec\theta_1\cdot \nolinebreak T_1$, and let $M_3$ be a pointwise slant
immersed submanifold of the almost Hermitian manifold $(M_2,g,J_2)$, with the slant function $\theta_2$.
Then, $M_3$ is a pointwise slant submanifold of $(M_1,g,J_1)$, with the slant function $$\tilde \theta_1=\arccos\Big(\cos\theta_1\cos\theta_2\Big).$$
\end{theorem}
\begin{proof}
For any $x\in M_3$, we have the orthogonal decomposition
$$T_xM_2=T_xM_3\oplus T^{\bot_2}_x M_3,$$
where $T^{\bot_2}_x M_3$ stands for the orthogonal complement of $T_xM_3$ in $T_xM_2$, and for any $v\in T_xM_3$, we denote
$$J_2v=T_2v+N_2v,$$
where $T_2v$ and $N_2v$ represent the tangential and the normal component of $J_2v$, respectively.

Since $M_3$ is also an
immersed submanifold of $(M_1,g,J_1)$, for any $v\in T_xM_3$, we denote $J_1v=\tilde T_1v+\tilde N_1v$ with $\tilde T_1v\in T_xM_3$ and $\tilde N_1v\in T^{\bot_1}_xM_3$, where $T^{\bot_1}_x M_3$ stands for the orthogonal complement of $T_xM_3$ in $T_xM_1$, and we have
\begin{align}
\tilde T_1v+\tilde N_1v&=J_1v=T_1v+N_1v\notag\\
&=\cos \theta_1(x) \cdot J_2v+N_1v\notag\\
&=\cos \theta_1(x)(T_2v+N_2v)+N_1v.\notag
\end{align}
Then,
$$
\Vert \tilde T_1v\Vert^2=\cos^2\theta_1(x)\Vert T_2v\Vert^2=\cos^2\theta_1(x)\cos^2\theta_2(x)\Vert v\Vert^2,$$
hence the conclusion.
\end{proof}

\begin{remark}\label{r1}
Since $$\cos\theta_1\cos\theta_2\leq \min\{\cos\theta_1,\cos\theta_2\},$$
we notice that $\tilde \theta_1\geq \max \{\theta_1,\theta_2\}$.
\end{remark}

\begin{example} \label{e3}
Let $M_1$ be the $8$-dimensional almost Hermitian manifold $\mathbb R^{8}$ with the Euclidean metric $g_0$ and with the almost complex structure $J_1$ defined by
$$J_1\Big(\frac{\partial}{\partial u_i}\Big)=-\frac{\partial}{\partial v_i}, \ \
J_1\Big(\frac{\partial}{\partial v_i}\Big)=\frac{\partial}{\partial u_i}$$
for $i=\overline{1,4}$, where we denoted by $(u_1,v_1,\dots,u_4,v_4)$ the canonical coordinates in $\mathbb R^8$.

Let $M_2$ be an immersed submanifold of $M_1$, defined by
$$F_1:\mathbb R^4\rightarrow \mathbb R^8,$$
$$F_1(x_1,x_2,x_3,x_4):= \Big(x_{1}+x_2, x_1-x_2,x_3+x_4, x_3-x_4,x_1,x_2,x_3,x_4\Big).$$

Then, $TM_2$ is spanned by the orthogonal vector fields:
$$X_{1} =\frac{\partial }{\partial u_{1}}+\frac{\partial }{\partial v_{1}}+\frac{\partial }{\partial u_{3}},\ \ X_{2} =\frac{\partial }{\partial u_{1}}-\frac{\partial }{\partial v_{1}}+\frac{\partial }{\partial v_{3}},$$
$$X_{3} =\frac{\partial }{\partial u_{2}}+\frac{\partial }{\partial v_{2}}+\frac{\partial }{\partial u_{4}},\ \ X_{4} =\frac{\partial }{\partial u_{2}}-\frac{\partial }{\partial v_{2}}+\frac{\partial }{\partial v_{4}}.$$

Applying $J_1$ to the base vector fields of $TM_2$, we get:
$$
J_1 X_{1} =\frac{\partial }{\partial u_{1}}-\frac{\partial }{\partial v_{1}}-\frac{\partial
}{\partial v_{3}},\ \
J_1 X_{2} =-\frac{\partial }{\partial u_{1}}-\frac{\partial
}{\partial v_{1}}+\frac{\partial }{\partial u_{3}},$$
$$J_1X_{3} =\frac{\partial }{\partial u_{2}}-\frac{\partial }{\partial v_{2}}-\frac{\partial }{\partial v_{4}},\ \
J_1X_{4} =-\frac{\partial }{\partial u_{2}}-\frac{\partial }{\partial v_{2}}+\frac{\partial }{\partial u_{4}}.$$

Then,
$$\Vert X_i\Vert=\sqrt{3}=\Vert J_1X_i\Vert, \ \ i=\overline{1,4},$$
$$g_0(J_1 X_1,X_2)=1 =g_0(J_1 X_3,X_4),$$$$g_0(X_1,J_1 X_2)=-1 =g_0(X_3,J_1 X_4),$$
$$g_0(J_1 X_1,X_3)=g_0(J_1 X_1,X_4)=g_0(J_1 X_2,X_3)=g_0(J_1 X_2,X_4)=0,
$$$$g_0(X_1,J_1 X_3)=g(X_1,J_1 X_4)=g_0(X_2,J_1 X_3)=g_0(X_2,J_1 X_4)=0;$$
hence, $M_2$ is a slant submanifold of $(\mathbb R^8, g_0,J_1)$, with the slant angle
$$\theta_{1}=\arccos \Big(\frac{1}{3}\Big).$$

We get: $$T_1X_1=\frac{1}{3}X_2, \ \ T_1X_2=-\frac{1}{3}X_1, \ \ T_1X_3=\frac{1}{3}X_4, \ \ T_1X_4=-\frac{1}{3}X_3,$$
and we define on $M_2$ the almost complex structure $J_2:=\sec \theta_1\cdot T_1$, i.e.,
$$J_2X_1=X_2, \ \ J_2X_2=-X_1, \ \ J_2X_3=X_4, \ \ J_2X_4=-X_3.$$

Let $M_3$ be an immersed submanifold of $\mathbb R^4$, defined by
$$F_2:\mathbb R^2\rightarrow \mathbb R^4,\ \ F_2(x_1,x_2):= \Big(x_1,x_1+x_2,x_1-x_2,x_2\Big),$$
and let $\tilde M_3$ be the submanifold of $M_2$ defined by the immersion
$$\tilde F_1:=F_1\circ F_2:\mathbb R^2\rightarrow \mathbb R^8,$$
$$\tilde F_1(x_1,x_2):=\Big(2x_{1}+x_2, -x_2,x_1, x_1-2x_2, x_1,x_1+x_2, x_1-x_2,x_2\Big).$$

Then, $T\tilde M_3$ is spanned by the orthogonal vector fields:
$$Z_{1} =2\frac{\partial }{\partial u_{1}}+\frac{\partial }{\partial u_{2}}+\frac{\partial }{\partial v_{2}}+\frac{\partial }{\partial u_{3}}+\frac{\partial }{\partial v_{3}}+\frac{\partial }{\partial u_{4}},$$
$$Z_{2} =\frac{\partial }{\partial u_{1}}-\frac{\partial }{\partial v_{1}}-2\frac{\partial }{\partial v_{2}}+\frac{\partial }{\partial v_{3}}-\frac{\partial }{\partial u_{4}}+\frac{\partial }{\partial v_{4}},$$
or, relative to the base $\{X_1, X_2, X_3, X_4\}$ of $TM_2$:
$$Z_1=X_1+X_2+X_3, \ \ Z_2=X_2-X_3+X_4.$$

Applying $J_2$ to to the base vector fields of $T\tilde M_3$, we get:
$$J_2 Z_1=-X_1+X_2+X_4, \ \ J_2 Z_2=-X_1-X_3-X_4.$$

Then,
$$\Vert Z_i\Vert=\sqrt{3}=\Vert J_2 Z_i\Vert, \ \ i={1,2},$$
$$\vert g_0(J_2 Z_1,Z_2)\vert=2=\vert g_0(J_2 Z_2,Z_1)\vert;$$
hence, $\tilde M_3$ is a slant submanifold of $(M_2,g_0,J_2)$, with the slant angle
$$\theta_{2}=\arccos \Big(\frac{2}{3}\Big).$$

On the other hand, $\tilde M_3$ is an immersed submanifold of $M_1$, with the immersion given by $\tilde F_1:\mathbb R^2\rightarrow \mathbb R^8$.

Then, $T\tilde M_3$ is spanned by the orthogonal vector fields $Z_1$ and $Z_2$ from above. Applying $J_1$ to this base of $T\tilde M_3$, we get:
\begin{align}
J_1 Z_{1} &=-2\frac{\partial }{\partial v_{1}}+\frac{\partial }{\partial u_{2}}-\frac{\partial }{\partial v_{2}}+\frac{\partial }{\partial u_{3}}-\frac{\partial }{\partial v_{3}}-\frac{\partial }{\partial v_{4}},\notag \\
J_1 Z_{2} &=-\frac{\partial }{\partial u_{1}}-\frac{\partial }{\partial v_{1}}-2\frac{\partial }{\partial u_{2}}+\frac{\partial }{\partial u_{3}}+\frac{\partial }{\partial u_{4}}+\frac{\partial }{\partial v_{4}}.\notag
\end{align}

Then,
$$\Vert Z_i\Vert=3=\Vert J_1Z_i\Vert, \ \ i={1,2},$$
$$\vert g_0(J_1 Z_1,Z_2)\vert =2=\vert g_0(J_1 Z_2,Z_1)\vert;$$
hence, $\tilde M_3$ is a slant submanifold of $(M_1, g_0,J_1)$, with the slant angle
$$\tilde \theta_1=\arccos \Big(\frac{2}{9}\Big),$$
and we notice that
$$\cos\tilde \theta_1=\cos \theta_1 \cdot \cos\theta_2.$$
\end{example}

\bigskip

For $\theta_2=\frac{\pi}{2}$ in Theorem \ref{p2}, we obtain
\begin{corollary}\label{c}
If $M_2$ is a proper pointwise slant
immersed submanifold of an almost Hermitian manifold $(M_1,g,J_1)$, with the slant function $\theta_1$,
and $M_3$ is an anti-invariant
immersed submanifold of the almost Hermitian manifold $(M_2,g,J_2)$, where $J_2:=\sec\theta_1\cdot T_1$, then $M_3$ is an anti-invariant submanifold of $(M_1,g,J_1)$.
\end{corollary}

\begin{example} \label{e4}
Let $(M_1, g_0,J_1)$ and $(M_2, g_0,J_2)$ be the almost Hermitian manifolds and $F_1$ the immersion defined in Example \ref{e3}.

Let $M_3$ be an immersed submanifold of $\mathbb R^4$, defined by
$$F_2:\mathbb R^2\rightarrow \mathbb R^4,\ \ F_2(x_1,x_2):= \Big(2x_1,x_2,2x_2,x_1\Big).$$

The image of $M_3$ through $F_1$ is $\tilde M_3:=F_1(M_3)$, which is an immersed submanifold of $M_2$. Denoting $\tilde F_1:=F_1\circ F_2$, we have $\tilde M_3=\tilde F_1(\mathbb R^2)$, where
$$\tilde F_1:\mathbb R^2\rightarrow \mathbb R^8,$$
$$\tilde F_1(x_1,x_2):= \Big(2x_{1}+x_2, 2x_1-x_2,x_1+2x_2, -x_1+2x_2, 2x_1,x_2, 2x_2,x_1\Big).$$

We denote by $(u_1,v_1,\dots, u_4,v_4)$ the canonical coordinates in $\mathbb R^8$.
Then, $T\tilde M_3$ is spanned by the orthogonal vector fields:
$$Z_{1} =2\frac{\partial }{\partial u_{1}}+2\frac{\partial }{\partial v_{1}}+\frac{\partial }{\partial u_{2}}-\frac{\partial }{\partial v_{2}}+2\frac{\partial }{\partial u_{3}}+\frac{\partial }{\partial v_{4}},$$
$$Z_{2} =\frac{\partial }{\partial u_{1}}-\frac{\partial }{\partial v_{1}}+2\frac{\partial }{\partial u_{2}}+2\frac{\partial }{\partial v_{2}}+\frac{\partial }{\partial v_{3}}+2\frac{\partial }{\partial u_{4}},$$
or, relative to the base $\{X_1, X_2, X_3, X_4\}$ of $TM_2$ from Example \ref{e3}:
$$Z_1=2 X_1+X_4, \ \ Z_2=X_2+2 X_3.$$

Applying $J_2$ to the base vector fields of $T\tilde M_3$, we get:
$$J_2 Z_{1} =2 X_2-X_3,\ \ J_2 Z_{2} =-X_1+2 X_4.$$

Then,
$$g_0(J_2 Z_1,Z_2) =0=g_0(J_2 Z_2,Z_1),$$
hence, $\tilde M_3$ is an anti-invariant submanifold of $M_2$.

On the other hand, $\tilde M_3$ is an immersed submanifold of $M_1$, with the immersion given by
$\tilde F_1:\mathbb R^2\rightarrow \mathbb R^8$, and $T\tilde M_3$ is spanned by the orthogonal vector fields $Z_{1}$ and $Z_{2}$ from above.

Applying $J_1$ to $Z_{1}$ and $Z_{2}$, we get:
\begin{align}
J_1 Z_{1} &=2\frac{\partial }{\partial u_{1}}-2\frac{\partial }{\partial v_{1}}-\frac{\partial }{\partial u_{2}}-\frac{\partial }{\partial v_{2}}-2\frac{\partial }{\partial v_{3}}+\frac{\partial }{\partial u_{4}},\notag \\
J_1 Z_{2} &=-\frac{\partial }{\partial u_{1}}-\frac{\partial }{\partial v_{1}}+2\frac{\partial }{\partial u_{2}}-2\frac{\partial }{\partial v_{2}}+\frac{\partial }{\partial u_{3}}-2\frac{\partial }{\partial v_{4}}.\notag
\end{align}

Then,
$$g_0(J_1 Z_1,Z_2)=0=g_0(J_1 Z_2,Z_1);$$
hence, $\tilde M_3$ is an anti-invariant submanifold of $(M_1, g_0,J_1)$.
\end{example}

\bigskip

We can state now a more general result.
Let $i\in \mathbb N^*$, and let $M_{i+1}$ be an
immersed submanifold of an almost Hermitian manifold $(M_i,g,J_i)$. Then, for any $x\in M_{i+1}$, we have the orthogonal decomposition
$$T_xM_{i}=T_xM_{i+1}\oplus T^{\bot}_x M_{i+1},$$
and for any $v\in T_xM_{i+1}$, we denote
$$J_iv=T_{i}v+N_{i}v,$$
where $T_{i}v$ and $N_{i}v$ represent the tangential and the normal component of $J_iv$, respectively.

\begin{theorem}\label{p3}
Let $(M_1,g,J_1)$ be an almost Hermitian manifold. Let $k\geq 2$, and let $M_{i+1}$ be a pointwise slant
immersed submanifold of the almost Hermitian manifold $(M_i,g,J_i)$, with the slant function $\theta_{i}$, $i=\overline{1,k}$, such that $\theta_{i}-\frac{\pi}{2}$ is nowhere zero on $M_{i+1}$, $i=\overline{1,k-1}$, where $J_{i}:=\sec\theta_{i-1}\cdot T_{i-1}$, $i=\overline{2,k}$. Then, $M_{k+1}$ is a pointwise slant submanifold of $(M_1,g,J_1)$, with the slant function $$\tilde \theta_{1}=\arccos\Big(\prod_{i=1}^{k}\cos \theta_{i}\Big).$$
\end{theorem}

\begin{remark}\label{r2}
Since $$\cos\tilde \theta_{1}=\prod_{i=1}^{k}\cos \theta_{i}\leq \min\{\cos\theta_{i}\ : \ i=\overline{1,k}\},$$
we notice that $\tilde \theta_{1}\geq \max \{\theta_1,\dots,\theta_{k}\}$.
\end{remark}

\bigskip

For $\theta_{k}=\frac{\pi}{2}$, we obtain
\begin{corollary}
If $(M_1,g,J_1)$ is an almost Hermitian manifold, $k\geq 2$, and $M_{i+1}$ is a pointwise slant
immersed submanifold of the almost Hermitian manifold $(M_i,g,J_i)$, $i=\overline{1,k}$, where $J_{i}:=\sec\theta_{i-1}\cdot T_{i-1}$, $i=\overline{2,k}$, with the slant function $\theta_{i}$ such that $\theta_{i}-\nolinebreak \frac{\pi}{2}$ is nowhere zero on $M_{i+1}$, $i=\overline{1,k-1}$, and $M_{k+1}$ is an anti-invariant
immersed submanifold of $M_k$, then $M_{k+1}$ is an anti-invariant submanifold of $(M_1,g,J_1)$.
\end{corollary}

\bigskip

From Theorem \ref{p2}, we can finally conclude
\begin{proposition}
The property of a submanifold to be pointwise slant (with respect to the naturally induced almost complex structure) is transitive on the class of proper pointwise slant
immersed submanifolds of almost Her\-mi\-tian manifolds.
\end{proposition}

\end{document}